\documentclass{amsart}

\usepackage{amsfonts,amsmath,amssymb,amsthm,amscd,amsxtra}
\usepackage{enumerate,verbatim}

\usepackage[all,2cell,ps]{xy}

\theoremstyle{plain} 
\newtheorem{theorem}{Theorem}[section]

\newtheorem{proposition}[theorem]{Proposition}
\newtheorem{corollary}[theorem]{Corollary}
\newtheorem{lemma}[theorem]{Lemma}

\theoremstyle{definition}

\newtheorem{example}[theorem]{Example}

\newtheorem{question}[theorem]{Question}
\newtheorem{remark}[theorem]{Remark}

\newtheorem{not&def}[theorem]{Notation and Definitions}
\newtheorem{notation}[theorem]{Notation}

\newtheorem{chunk}[theorem]{\hspace*{-1.065ex}\bf}

\numberwithin{equation}{theorem}

\newcommand{\lra}{\longrightarrow}

\newcommand{\xra}{\xrightarrow}

\newcommand{\fm}{\mathfrak{m}}

\newcommand{\fp}{\mathfrak{p}}

\newcommand{\sft}{\mathsf{t}}

\newcommand{\ZZ}{\mathbb{Z}}

\DeclareMathOperator{\ann}{ann}
\DeclareMathOperator{\Ass}{Ass}

\DeclareMathOperator{\depth}{depth}

\DeclareMathOperator{\Ext}{Ext}

\DeclareMathOperator{\hh}{H}
\DeclareMathOperator{\Hom}{Hom}

\DeclareMathOperator{\pd}{pd}

\DeclareMathOperator{\rank}{rank}

\DeclareMathOperator{\Spec}{Spec}

\DeclareMathOperator{\Tor}{Tor}

\newcommand{\vf}{\varphi}

\newcommand{\ul}{\underline}
\newcommand{\ov}{\overline}

\newcommand{\tensor}{\otimes}

\newcommand{\tf}[2]{{\boldsymbol\bot}_{#1}{#2}}
\newcommand{\tp}[2]{{\boldsymbol\top}_{\hskip-2pt #1}{#2}}

\def\urltilda{\kern -.15em\lower .7ex\hbox{\~{}}\kern .04em}
\def\urldot{\kern -.10em.\kern -.10em}\def\urlhttp{http\kern -.10em\lower -.1ex
\hbox{:}\kern -.12em\lower 0ex\hbox{/}\kern -.18em\lower 0ex\hbox{/}}

\begin{document}

\title[Torsion in tensor products]
{Torsion in tensor powers of modules}
\author[Celikbas, Iyengar, Piepmeyer, Wiegand]
{Olgur Celikbas, Srikanth B. Iyengar, \\ Greg Piepmeyer, and Roger Wiegand}

\address{Olgur Celikbas \\
Department of Mathematics \\
University of Missouri \\
Columbia, MO 65211,USA}
\email{celikbaso@missouri.edu}

\address{Srikanth B. Iyengar \\
Department of Mathematics \\
University of Nebraska\\
Lincoln, NE 68588, USA}
\email{s.b.iyengar@unl.edu}

\address{Greg Piepmeyer \\
Columbia Basin College \\
Pasco, WA 99301, USA}
\email{gpiepmeyer@columbiabasin.edu}

\address{Roger Wiegand\\
Department of Mathematics \\
University of Nebraska \\
Lincoln, NE 68588, USA}
\email{rwiegand1@math.unl.edu}



\thanks{Part of this material is based upon work supported by the
  National Science Foundation under Grant No. 0932078 000, while
  SBI and RW were in residence at the Mathematical Science Research
  Institute in Berkeley, California, during the Fall semester of 2013.
  SBI partly supported by NSF grant DMS-1201889 and a Simons
  Fellowship; RW partly supported by a Simons Collaboration Grant}

\date{\today}

\begin{abstract}

 Tensor products usually have nonzero torsion.  This is a central theme of Auslander's paper \cite{Au};
 the theme continues in the work of Huneke and Wiegand \cite{HW1}--\cite{HW2}. The main focus in this note is on tensor powers of a finitely generated module  over a local ring.
 Also, we study torsion-free modules $N$ with the property that $M\otimes_RN$ has nonzero torsion unless $M$ is very special. An important example of such a module $N$ is the
Frobenius power  ${}^{p^e}\!R$ over a complete intersection domain $R$ of characteristic $p>0$.
\end{abstract}

\maketitle{}

\section{Introduction}

In a 1961 paper \cite{Au}, Auslander studied torsion in tensor
products of nonzero finitely generated modules $M$ and $N$ over
unramified regular local rings $R$.  Under the assumption that
$M\otimes_RN$ is torsion-free, he proved:
\begin{enumerate}
\item $M$ and $N$ must be torsion-free, and
\item $M$ and $N$ are \emph{Tor-independent}, that is, $\Tor_i^R(M,N)=
  0$ for all $i\ge 1$.\label{pg:tor independent}
\end{enumerate}
The two conclusions are cleverly intertwined in his proof,
which we revisit in Section~\ref{sec:Frobenius} of
the present paper.    We show, over a reduced complete intersection ring $R$ of positive characteristic $p$, that
$M\otimes_R{}^{\vf^e}\!R$ is torsion-free if and only if $M$ is torsion-free and of finite projective dimension,
in which case $\Tor^R_i(M,  {}^{\vf^e}\!R) = 0$ for all $i\ge 1$.  (Here $\vf:R\to R$ is the Frobenius endomorphism and
${}^{\vf^e}\!R$ is the module obtained from $R$ by restriction of scalars along $\vf^e$.)
When $R$ is F-finite, we obtain a criterion for regularity:
$R$ is regular if and only if $({}^{\vf^e}\!M)\otimes_R{}^{\vf^e}\!R$
torsion-free for some (equivalently, every) nonzero finitely generated $R$-module $M$.

Our main results are in Section~\ref{sec:tensor-powers}, where we study torsion in tensor powers.  We obtain detailed information on annihilators of elements in $\otimes^n_RM$ and draw several conclusions.  Suppose, for example, that $\ul r = r_1,\dots,r_d$ is a regular sequence in $R$ and $M$ is the cokernel of the $d\times1$ matrix $[\ul r]^{\text{t}}$.  We show in Theorem~\ref{thm:annihilator} that $\otimes^t_RM$ is torsion-free if and only $t\le d$.  This result should be compared with Auslander's observation \cite[p. 638]{Au} that the same holds when $M$ is the $d-1^{\text{st}}$ syzygy of a module of projective dimension $d$ over a $d$-dimensional regular local ring.  Cf. also \cite[Proposition~3.1]{HW2}.
If $R$ is local, the ``only if'' direction holds much more generally:  if  we write $M$ as the cokernel of an $m\times n$  matrix $\theta$ with entries in the maximal ideal of $R$, and if some entry of $\theta$ is a non-zerodivisor, we show in Theorem~\ref{thm:tors-in-power} that $\otimes_R^tM$ has nonzero torsion for every $t\ge m$.

Throughout this paper, $R$ is a commutative, Noetherian ring.

\section{Torsion in tensor powers}
\label{sec:tensor-powers}
In this section we establish results on annihilators of elements in
tensor powers of modules.

\begin{notation}
\label{not:shuffle}
Given elements $\ul m:=m_{1},\dots,m_{d}$ in an $R$-module $M$, we
consider the element in $\tensor^d_RM$ defined by
\[
 \tau(\ul m):= \, \sum_{\sigma\in S_{d}}\mathrm{sign}(\sigma)\,
 m_{\sigma(1)}\otimes\dots\otimes m_{\sigma(d)}\,.
\]
\end{notation}

\begin{proposition}
\label{prop:tors-elt}
Let $M$ be an $R$-module. If elements $m_1,\dots,m_d$ in $M$ and
$r_1,\dots,r_d $ in $R$ satisfy
\begin{equation}
\label{eq:relation}
 r_1m_1+\dots+r_dm_d = 0,
\end{equation}
then $(r_{1},\dots,r_{d})\cdot \tau(\ul m)=0$ in $\otimes_R^dM$.
\end{proposition}

\begin{proof}
\label{pf:shuffle}
The \emph{twisted shuffle product} gives the graded $R$-algebra
$\bigoplus_{n\ge 0}\otimes_R^nM$ a strictly skew-commutative
structure; see \cite[Chapter X, (12.4)]{MacLane}.  Strictly skew-commutative
means that for any $a\in \otimes_R^iM$ and $b\in \otimes^j_RM$, there
are equalities
\[
a\star b = (-1)^{ij}b\star a\,, \quad\text{and}\quad a\star a =
0\text{ when $i$ is odd.}
\]
By definition of the shuffle product, $ \tau(\ul m)=m_{1}\star \cdots
\star m_{d}$.  Thus for each $j$ we have
\[
\begin{split}
r_j\cdot \tau(\ul m) & = m_1 \star \cdots \star m_{j-1} \star r_j m_j \star
m_{j+1} \star \cdots \star m_n \\
& = -\sum_{i \neq j} r_i (m_1 \star \cdots \star
m_{j-1} \star m_i \star m_{j+1} \star \cdots \star m_n) = - \sum_{i \neq j}
r_i 0 = 0.  \qedhere
\end{split}
\]
\end{proof}

There is a ``universal'' source for the element $\tau(\ul m)$ in the
following sense:

\begin{remark}
 Consider the polynomial ring $\ZZ[\ul x]$ on indeterminates $\ul
 x:=x_{1},\dots,x_{d}$, and let $U$ be the $\ZZ[\ul x]$-module with
 presentation
\[
0\lra \ZZ[\ul x] \xra{\ [x_{1},\dots,x_{d}]^{\sft}\ } {\ZZ[\ul
  x]}^{d}\lra U\lra 0\,.
\]
Let $u_{1},\dots,u_{d}$ the the generators of $U$ corresponding to the
standard basis for $\ZZ[\ul x]^{d}$, so that
$x_{1}u_{1}+\cdots+x_{d}u_{d}=0$, i.e., \( \ul x \) and \( \ul u \)
satisfy \eqref{eq:relation}.  Then $\ann_{\ZZ[\ul x]}\tau(\ul
u)\supseteq (\ul x)$ by Proposition~\ref{prop:tors-elt}; we shall see, in
Theorem~\ref{thm:annihilator} below, that in fact $\ann_{\ZZ[\ul x]}\tau(\ul
u) = (\ul x)$.

Given any $R$-module $M$ with a syzygy relation \eqref{eq:relation},
consider the ring homomorphism $\ZZ[\ul x] \to R$ taking
$x_{i}$ to $r_{i}$, for each $i$, and extending the structure
homomorphism $\ZZ\to R$. The hypothesis on $M$ implies that there is a
homomorphism of $\ZZ[\ul x]$-modules
\[
f\colon U \lra M\quad\text{with $f(u_{i})=m_{i}$ for $i=1,\dots,d$.}
\]
Under the induced map $\tensor^df\colon \tensor^d_{\ZZ[\ul x]}U\to
\tensor^d_RM$, the element $\tau(\ul u)$ maps to $\tau(\ul m)$.
 \end{remark}

This remark prompts the discussion below, culminating in Theorem~\ref{thm:annihilator}.
First we review some notions regarding depth. For details, see \cite[Chapter 1]{BH}.

\begin{chunk} \label{depthchunk}
  \textbf{Depth.}  Let $M$ be a finitely generated $R$-module and $I$ an
  ideal of $R$ satisfying $IM\ne M$. The \emph{$I$-depth} of
  $M$ is the number
\begin{equation*}
\label{eq:depth}
\depth_{R}(I,M)= \inf\{n\geq 0\mid \Ext_{R}^{n}(R/I,M)\ne 0\}\,.
\end{equation*}
The $I$-depth of $M$ is always finite and is equal to the length of every maximal
 $M$-regular sequence in $I$.

If $\ul x:=x_{1},\dots,x_{d}$ is a sequence of elements in $R$, and
$K$ is the Koszul complex on $\ul x$, then the $(\ul x)$-depth of $M$
may be computed from its Koszul homology:
\[
\depth_{R}((\ul x),M)
= d - \sup\{i\geq 0\mid \hh_{i}(K
\otimes_{R}M)\ne 0\}
\]
This is the \emph{depth sensitivity} of the Koszul complex.

Suppose now that $\ul x$ is $R$-regular. Then $K$ is a free resolution of $R/(\ul x)$,
and hence $\hh_{*}(K\otimes_{R}M) \cong \Tor^R_*(R/(\ul x),M)$.  In this case, we have
\begin{equation}\label{eq:depth-Tor}
\depth_{R}((\ul x),M) = d - \sup\{i\geq 0\mid \Tor^{R}_{i}(R/(\ul x),M)
 \ne 0\}\,.
\end{equation}

If $R$ is local  with maximal ideal $\fm$,
we write $\depth_{R}M$ for the $\fm$-depth of $M$ and call it the
\emph{depth} of $M$.
\end{chunk}

\begin{chunk}{\bf A Koszul syzygy module.}
\label{Kos-module}
Let $R$ be a Noetherian ring and $\ul r:=r_{1},\dots,r_{d}$ a regular
sequence in $R$ with $(\ul r)\ne R$.  Consider the complex
\[
F:=\quad 0\lra R\xra{\ [r_{1},\dots,r_{d}]^{\sft}\ } R^{d}\lra 0
\]
concentrated in degrees $0$ and $1$. Set $M=\hh_{0}(F)$; as $r_1$ is a
non-zerodivisor, $F$ is a free resolution of $M$.
\end{chunk}

\begin{lemma}
\label{lem:Koszul-torsion}
Let $M$, $d$, and $F$ be as in \rm{2.5}.  For each $n=1,\dots, d$, the following statements hold:
\begin{enumerate}[{\quad\rm(1)}]
\item $M$ and $\tensor_R^{n-1}M$ are $\Tor$-independent.
\item $\tensor^n_RF$ is a free resolution of $\tensor_R^nM$, and
  $\pd_{R}(\tensor_R^nM) = n$.
\end{enumerate}
\end{lemma}

\begin{proof}
The base case is $n=1$, and then (1) and (2) are clear.  Fix an
integer $n$ with $2\le n \le d$, and assume these statements hold for
all integers $\le n-1$. Set $I=(\ul r)$.  Since $\tensor^{n-1}_{R}F$ is a free resolution of
$\tensor_{R}^{n-1}M$, we have
\[
\Tor^{R}_{*}(R/I,\tensor_R^{n-1}M) =
\hh_{*}((R/I)\otimes_{R}(\tensor_R^{n-1}F)) \cong
(\tensor_R^{n-1}{((R/I)\otimes_{R}F)})_{*}\,,
\]
where the last isomorphism holds because the complex in question has zero differential. In particular, $\Tor^{R}_{n-1}(R/I,\tensor_R^{n-1}M)\cong R/I \ne 0$, so that
\begin{equation}
\label{eq:last-tor}
\sup\{i\geq 0\mid \Tor^{R}_{i}(R/(\ul r),\tensor_{R}^{n-1}M)\ne 0\}=n-1\,.
\end{equation}
We can now complete the induction step.

(1) The induction hypothesis implies that $\tensor^{n-1}_{R}F$ is a free
resolution of $\tensor_{R}^{n-1}M$, so \eqref{eq:last-tor} and
\eqref{eq:depth-Tor} show that
\begin{equation}\label{eq:pos-depth}
\depth_R(I, \otimes_R^{n-1} M) = d - (n-1) \ge 1\,.
\end{equation}
Moreover, $\Tor^{R}_{*}(M,\tensor_R^{n-1}M)$ is the homology of
the complex
\[
F\otimes_R(\otimes_{R}^{n-1}M): \qquad 0\lra \tensor_R^{n-1}M
\xra{\ [\ul r]^{\sft} } (\tensor_R^{n-1}M)^{d}\lra 0
\]
(concentrated in degrees $0$ and $1$).
By \eqref{eq:pos-depth}, some $r_i$ is a non-zerodivisor on $\otimes_R^{n-1}M$,
and it follows that $M$ and $\tensor_R^{n-1}M$ are $\Tor$-independent.

\medskip

(2) By hypothesis, $F$ and $\tensor_R^{n-1}F$ are free resolutions of
$M$ and $\tensor_R^{n-1}M$, respectively. We have already proved, in
(1), that these modules are $\Tor$-independent, so the complex
$F\otimes_{R}(\tensor_R^{n-1}F)$, that is to say, $\tensor_R^nF$, is a
free resolution of $\tensor_R^nM$. In particular,
$\pd_{R}(\tensor_R^nM) \leq n$; that equality holds follows from
\eqref{eq:last-tor}.
\end{proof}

\begin{chunk} \label{tor-sub}
  \textbf{Torsion submodule.}  Let $\text{Q}(R)$ be the total quotient
  ring of $R$.  The \emph{torsion submodule} $\tp R M$ of $M$ is the
  kernel of the natural homomorphism $M\to \text{Q}(R)\otimes_RM$.
  The inclusion $\tp RM\subseteq M$ gives rise to an exact sequence
\begin{equation}
\label{eq:tp}
0\lra \tp RM \lra M\lra \tf RM\lra 0\,.
\end{equation}
The module $M$ is \emph{torsion} if $\tp RM=M$ (that is, $M_\fp = 0$
for each $\fp\in \Ass(R)$), and $M$ is \emph{torsion-free} if $\tp RM
= 0$.  Thus $M$ is torsion-free if and only if $\bigcup\Ass M \subseteq \bigcup \Ass R$.  The
stronger condition, that
$\Ass M \subseteq \Ass R$, is therefore a sufficient condition for $M$ to be torsion-free.  We will
invoke this criterion twice in the proof of the next theorem.
\end{chunk}

Part (1) of the next result is reminiscent of Auslander's discussion on p. 638 of \cite{Au}.
Cf. also \cite[Proposition 3.1]{HW2}.

\begin{theorem}
\label{thm:annihilator}
Let $M$ and $\ul r$ be as in \emph{\ref{Kos-module}}.  The following
statements hold:
\begin{enumerate}[{\quad\rm(1)}]
\item
$\tensor_{R}^{n}M$ is torsion-free if and only if $n\le d-1$.
\item
The element $\tau(\ul m)$ in $\tensor_{R}^{d}M$ satisfies
$\ann_{R}\tau(\ul m)= (\ul r)$.
\item
The map $R/(\ul r) \to \tensor_{R}^{d}M$ of $R$-modules with $1\mapsto
\tau(\ul m)$ induces a splitting
\[
\tensor_{R}^{d}M\cong (R/(\ul r)) \bigoplus W
\]
where $W$ is torsion-free; in particular, we have
\[
 \Hom_{R}(R/(\ul r), \tensor_R^dM)=
 R \tau(\ul m)\ne 0\,.
\]
\end{enumerate}
\end{theorem}

\begin{proof}
Set $I=(\ul r)$, let $n\le d-1$,  and fix a prime $\fp\in\Ass(\otimes^n_RM)$. If $I \subseteq \fp$, it follows from Lemma~\ref{lem:Koszul-torsion} that
$(\otimes_{R}^{n}F)_{\fp}$ is a minimal free resolution of $(\tensor_R^{n}M)_{\fp}$; therefore
\[
\depth_{R_{\fp}}(\tensor_{R}^{n} M)_{\fp} = \depth R_{\fp} - n \geq d
- n \geq 1,
\]
contradiction.
Thus $I \not\subseteq \fp$ and then the $R_{\fp}$-module $M_{\fp}$ is a nonzero free module; hence so is
$(\tensor_R^{n}M)_{\fp}$. Therefore $\depth R_\fp = \depth_{R_\fp}(\otimes^n_RM)_\fp$.
We have shown that $\Ass(\otimes^n_RM) \subseteq \Ass R$, and hence that $\otimes_R^nM$ is torsion-free.
The ``only if'' direction of (1) will follow from (3).

\medskip

As for parts (2) and (3), by construction $r_{1}m_{1}+\cdots +
r_{d}m_{d}=0$, so Proposition~\ref{prop:tors-elt} gives an inclusion
$I \subseteq \ann_{R}\tau(\ul m)$.  The reverse inclusion will follow,
once we ascertain that the map in (3) splits. Consider the
homomorphisms of $R$-modules
\[
\tensor_{R}^{d}(F_{0}) \twoheadrightarrow \tensor_{R}^{d}M
\twoheadrightarrow (\tensor_{R}^{d}M)\otimes_{R}R/I \cong
\hh_{0}((\tensor_{R}^{d}F)\otimes_{R} R/I) =
\tensor_{R}^{d}(F_{0}\otimes_{R}R/I) \,,
\]
where the surjections are the natural ones; the isomorphism holds
because $\tensor_{R}^{d}F$ is a free resolution of $\tensor_{R}^{d}M$,
and the equality holds because the differential on $F$ has its
image in  $I F$. Let $\ul e = e_1,\dots,e_d $ be the
standard basis for $F_{0}=R^{d}$, in \ref{Kos-module}, and let ${\ul
  e}'$ be the induced basis of the free $R/I$-module
$F_{0}\otimes_{R}R/I$. Under the composite map, the element $\tau(\ul
e)$ maps to $\tau(\ul e')$, and $\{\tau(\ul e')\}$ extends to a basis of the
$R/I$-module $\tensor_{R}^{d}(F_{0}\otimes_{R}(R/I))$. Since $\tau(\ul
e)$ maps to $\tau(\ul m)$ in $\tensor_{R}^{d}M$, the map in (2) splits
and gives a decomposition
\[
\tensor_{R}^{d} M\cong (R/I) \bigoplus W\,.
\]
It remains to verify that $W$ is torsion-free; given the decomposition
above, the other parts of (3) are a consequence of this fact.

For $\fp\in\Spec R$ with $I\not\subseteq \fp$, the
$R_{\fp}$-module $M_{\fp}$ free, and hence so is $W_{\fp}$.
Assume now that $I\subseteq \fp$. The Koszul complex on $\ul r$, viewed
as elements in $R_{\fp}$, is a minimal resolution of $(R/I)_{\fp}$,
and so it is a direct summand of $(\otimes_{R}^{n}F)_{\fp}$, the
minimal free resolution of $(\tensor_R^{n}M)_{\fp}$. The ranks of the
free modules in the top degree, $d$, of these complexes coincide (and
equal $1$), whence $\pd_{R_{\fp}}W_{\fp}\le d-1$ and
\[
\depth_{R_{\fp}}W_{\fp} = \depth R_{\fp} - \pd_{R_{\fp}}W_{\fp} \ge 1
\]
These observations show that $\Ass W \subseteq \Ass R$, so
$W$ is torsion-free as claimed.
 \end{proof}

\subsection*{Local rings}
Next we focus on local rings, where the preceding results can be
strengthened to some extent.

\begin{lemma}
\label{lem:tau-not0}
Let $M$ be a finitely generated module over a local ring $(R,\fm)$,
and let $m_1,\dots,m_d\in M$.  If the images of $\{m_1,\dots,m_d\}$ in
$M/\fm M$ are linearly independent, then $\tau(\ul m)$ is not in $\fm
(\tensor_R^dM)$.
\end{lemma}

 \begin{proof}
Let \( m_i'\) be the image of \( m_i \) in the \( k\)-vector space
\(M/\fm M \).  Since $\{m_1',\dots,m_d'\}$ is linearly independent, \(
\tau(\ul{m}') \neq 0 \).  Hence \( \tau(\ul{m}) \notin \fm
(\tensor^d_R M) \).  \end{proof}

Given an $R$-module $M$, we write $I(M)$ for the ideal $(r_{ij})$
defined by the entries in a matrix in some minimal presentation
\[
R^{\mu}\xra{\ [r_{ij}]\ }R^{\nu}\lra M\lra 0\quad\text{where
  $\nu=\nu_{R}(M)$.}
\]
This ideal is independent of the presentation. Moreover,
 $I(M)$ contains a non-zerodivisor if and only if,
over $\text{Q}(R)$, the total quotient ring of $R$, the module
$\text{Q}(R)\otimes_{R}M$ can be generated by fewer than $\nu$
elements.  To see this we note that, since $\text{Q}(R)$ is semilocal, the module
$\text{Q}(R)\otimes_{R}M$ {\em needs} $\nu$ generators if and only if
$\nu_{R_\fp}M_\fp = \nu$ for some $\fp\in \Ass R$; moreover, $\nu_{R_\fp}M_\fp = \nu$
if and only if the presentation remains minimal when localized at $\fp$, that is, if and only if
 $I(M) \subseteq \fp$.  Thus
$\text{Q}(R)\otimes_{R}M$ needs $\nu$ generators if and only if $I(M) \subseteq \fp$
for some $\fp\in \Ass R$,
that is, if and only if $I(M)$ consists of zerodivisors.

 Recall that $M$ is said to {\em have rank} \label{rk page}
$r$ if $\text{Q}(R)\otimes_{R}M$ is free over $\text{Q}(R)$ of rank
$r$; see \cite[Proposition 1.4.3]{BH} for different characterizations
of this property.

\begin{theorem}
\label{thm:tors-in-power}
Let $R$ be a local ring and $M$ a nonzero finitely generated
$R$-module satisfying one of the following conditions:
\begin{enumerate}[\quad\rm(1)]
\item $I(M)$ contains a non-zerodivisor; in this case, set
  $b=\nu_{R}(M)$; or
\item $M$ has rank; in this case, set $b=\rank_{R}(M)+1$.
\end{enumerate}
If $M$ is not free, then for each nonzero finitely generated
$R$-module $N$ one has
\[
\tp R{((\tensor_R^nM)\otimes_{R}N)}\ne 0 \quad\text{for each $n\geq
  b$}\,.
\]
\end{theorem}

\begin{proof}
It suffices to prove the statement for $n=b$, since
\[
(\tensor_R^nM)\otimes_{R}N \cong
(\tensor_R^bM)\otimes_{R}((\tensor_R^{n-b}M)\otimes_{R}N)\,,
\]
and $N\ne 0$ implies $(\tensor_R^iM)\otimes_RN\ne 0$ for each $i\geq
0$, by Nakayama's lemma.

(1) Let $m_{1},\dots,m_{b}$ be a minimal generating set for the $R$-module $M$. The element $\tau(\ul m)$ in $\tensor_R^bM$
is annihilated by $I(M)$, by Proposition~\ref{prop:tors-elt}, and is not in $\fm (\tensor_R^bM)$, by Lemma~\ref{lem:tau-not0}.
It follows that, for each $x$ in $N\setminus \fm N$, the element $\tau(\ul m)\otimes x$ in $(\tensor_R^bM)\otimes_{R}N$ is nonzero 
and is annihilated by $I(M)$, and hence is in the torsion submodule; this where the hypothesis that $I(M)$ contains a non-zerodivisor is used.

(2)  We claim that there exists a syzygy relation \eqref{eq:relation} with $\ul m$ a minimal generating set for $M$, $(\ul r)\subseteq\fm$, and
some $r_i$ a non-zerodivisor.

Indeed, $\nu_{R}(M)\ge b$ since $M$ is not free. Choose elements $m_{1},\dots,m_{b}$ that form part of a minimal generating set for $M$ and such that $m_{1},\dots,m_{b-1}$ form a basis for $\text{Q}(R)\otimes_{R}M$ over $\text{Q}(R)$. Then there is a syzygy relation as in \eqref{eq:relation} in which $r_{b}$ is a non-zerodivisor.

The element $\tau(\ul m)$ in $\tensor_R^bM$ is annihilated by $(\ul r)$, by Proposition~\ref{prop:tors-elt}, and is not in $\fm (\tensor_R^bM)$, by Lemma~\ref{lem:tau-not0}. Since $(\ul r)$ has a non-zerodivisor,  it follows as in (1) that the torsion submodule of $(\tensor_R^bM)\otimes_{R}N$ is nonzero.
\end{proof}

We learned recently that in 2011, in response to a query on MathOverflow,
David Speyer gave a proof (quite similar to ours) of part (1) when $R$ is a domain.
(See http://mathoverflow.net/questions/73120/torsion-free-tensor-powers.)

\medskip

One cannot always expect torsion in tensor powers of non-free modules:

\begin{example}
\label{ex:node1}
Let $R=k[[x,y]]/(xy)$, where $k$ is a field. The torsion-free $R$-module $M:=R/(x)$ is not free; however $\tensor_R^nM$ is isomorphic to $R/(x)$ for every $n\ge 1$, and hence is torsion-free.
\end{example}

The preceding results bring to the fore the following:

\begin{question}
\label{quest}
Let $R$ be a local domain. Is there an integer $b$, depending only on
$R$, such that $\tensor^n_RM$ has torsion for every finitely generated
non-free $R$-module $M$ and every integer $n\ge b$?
\end{question}

The condition that $R$ be a domain is to avoid the situation of
Example~\ref{ex:node1}. When $R$ is regular, one can take $b=\dim R$,
by results of Auslander~\cite[Theorem 3.2]{Au} and
Lichtenbaum~\cite[Corollary 3]{Li}.

\section{Torsion ``carriers''}
\label{sec:Frobenius}
Some modules, even though they are torsion-free, usually generate torsion in tensor products.  For example, over a local ring $(R,\fm, k)$ of positive depth, the maximal ideal $\fm$ is such a module:  for any finitely generated non-free $R$-module $M$, the tensor product $\fm\otimes_RM$ has torsion.  To see this, observe that the short exact
sequence
\[
0 \to \fm \to R \to k \to 0
\]
yields an injection from the
torsion module $\Tor_1^R( k,M)$ into $\fm\otimes_RM$; moreover, $\Tor_1^R(k,M) \ne 0$ because $M$ is not free.

We will give two more  examples of torsion carriers:
the integral closure $\ov R$ of a one-dimensional analytically unramified ring $R$, and the  Frobenius powers ${}^{\vf^e}\!R$ of a complete intersection $R$ of characteristic $p$.  Recall that a local ring is {\em analytically unramified} provided its completion is reduced.  If $R$ is one-dimensional, an equivalent condition  is that  $R$ be Cohen-Macaulay with
 finitely generated integral closure $\ov R$, \cite[Theorem 4.6]{LW}.

\begin{theorem}\label{thm:integral-closure} Let $R$ be a one-dimensional analytically unramified
local ring, and let  $\ov R$ be the integral closure of $R$ in its total quotient ring. If $M$ is a finitely generated $R$-module
for which $\ov R\otimes_RM$ is torsion-free, then $M$ is free.
\end{theorem}
\begin{proof}  Let $\fp_1,\dots, \fp_s$ be the minimal
prime ideals of $R$, and for each $i$ let $r_i$ be the dimension of the $R_{\fp_i}$-vector space
$M_{\fp_i}$.   Put $n = \nu_R M$, the minimal number of generators of the $R$-module $M$, and choose an
exact sequence
\[
0 \to K \to R^{(n)} \to M \to 0\,.
\]
If we can show that $r_i = n$ for each $i$, we'll know that $K$ is torsion and hence zero, and we'll be done.

Put $D_i = \overline{R/\fp_i}$, the integral closure of the domain
$R/\fp_i$.  Since $R$ is reduced, we have inclusions
\[
R \hookrightarrow \prod_{i=1}^sR/\fp_i \hookrightarrow
\prod_{i=1}^sD_i \hookrightarrow
\prod_{i=1}^s\text{Q}(R/\fp_i) = \text{Q}(R)\,.
\]
We see that
$\ov R = \prod_{i=1}^sD_i$\,; moreover,
each $D_i$ is a semilocal Dedekind domain and therefore a principal ideal domain.
Since $\ov R\otimes_RM$ is torsion-free, it is projective,
in fact free of rank $r_i$ on the component $D_i$.  Therefore, setting $e_i = \nu_RD_i$,  we have the equations
\[
  r_1e_1+\dots + r_se_s = \nu_R(\ov R\otimes_RM) = (\nu_R\ov R)\cdot (\nu_RM) = (e_1 + \dots + e_s)n.
  \]
Since $r_i\le n$ for each $i$, it follows from these equations that $r_i = n$ for each $i$.
\end{proof}

Let $R$ be a Noetherian ring of positive characteristic $p$ and
${\vf}\colon R\lra R$ the Frobenius endomorphism $r\mapsto r^p$.
Given an $R$-module $M$ and a positive
integer $e$, we write ${}^{\vf^e}\!M$ for the $R$-module obtained
from $M$ by restriction of scalars along $\vf^e$; thus $r\cdot m =
r^{p^{e}}m$ for $r\in R$ and $m\in M$. Observe that $M$ is
torsion-free if and only if ${}^{\vf^e}\!M$ is torsion-free for some
(equivalently, all) $e\ge1$.  Following \cite{PS}, we write $F^e(M)$ for the tensor
product $M\otimes_R{}^{\vf^e}\!R$.  One views $F^e(M)$ as a {\em right} $R$-module:
the action of $R$ on $F^e(M)$ comes from the right (ordinary)
action of $R$ on ${}^{\vf^e}\!R$.  Thus $F^e(R) \cong R$ as $R$-modules,
and it follows that $F^e(M)$ is finitely generated if $M$ is finitely generated.

The following result follows immediately from \cite[Corollary 1.10]{PS}:

\begin{theorem}[Peskine and Szpiro, 1973]
\label{rmk:PS} 
Let $R$ be a local ring of  characteristic $p$, and let $M$ be a finitely generated $R$-module.  If $M$ has finite projective dimension, then $\Tor_i^R(M, {}^{\vf^e}\!R) = 0$ for all $e\ge1$ and all $i\ge 1$.
\end{theorem}

The converse of Theorem~\ref{rmk:PS}  is  true and was proved by Herzog \cite[Theorem 3.1]{Her}.
For complete intersections, the following strong converse was proved by Avramov and Miller \cite[Theorem]{AM}:

\begin{theorem}[Avramov and Miller, 2001]\label{rmk:AM} Let $(R,\fm)$ be a complete intersection
of characteristic $p$, and let $M$ be a finitely generated $R$-module.
If $\Tor_i^R(M, {}^{\vf^e}\!R) = 0$ for some $e\ge1$ and some $i\ge 1$, then $M$ has finite projective dimension.
\end{theorem}

The proof that (1)$\implies$(2) in the next theorem follows many of the same steps Auslander used in his proof of
 \cite[Lemma 3.1]{Au}.  The main differences are that we have to allow for the possibility that $^{\vf^e}\!R$ is not finitely generated, and that we appeal to Theorems~\ref{rmk:PS} and \ref{rmk:AM} for a replacement of  rigidity of Tor over regular local rings.  Recall that a module $M$ is {\em generically free} provided $M_\fp$ is a free $R_\fp$-module for each $\fp\in \Ass R$.

\begin{theorem}\label{thm:carrier}
Let $(R,\fm)$ be a complete intersection of characteristic $p$ and
$M$ a finitely generated, generically free $R$-module. Fix a positive integer $e$.
The following  conditions are
equivalent:
\begin{enumerate}[\quad\rm(1)]
\item $F^e(M)$ is torsion-free
\item $M$ is torsion-free and of finite projective dimension.
\end{enumerate}
\end{theorem}

\begin{proof}

Suppose (1) holds, and apply  $-\otimes_R{}^{\vf^e}\!R$ to the short exact sequence \eqref{eq:tp},
getting an exact sequence

\[
 F^e(\tp RM)\overset {\alpha} {\lra}  F^e(M)
\overset {\beta} {\lra} F^e(\tf RM) \lra 0\,.
\]
Since $F^e(\tp RM)$ is torsion and  $F^e(M)$ is torsion-free, we see that $\alpha = 0$, whence $\beta $ is an isomorphism.  In particular $ F^e(\tf RM)$ is torsion-free.  Next, consider the universal pushforward \cite[\S1]{HJW}:
\begin{equation}\label{eq:push}
0\lra \tf RM \lra R^{(m)} \lra N \lra 0\,.
\end{equation}
Applying $-\otimes_R {}^{\vf^e}\!R$ to this sequence, we obtain an injection

\[
\Tor_1^R(N,{}^{\vf^e}\!R) \hookrightarrow F^e(\tf RM)\,.
\]
Now $\tf RM$ is clearly generically free, and from the construction of the universal pushforward \cite[\S1]{HJW}
one checks that $N$ is generically free as well.  It follows that $\Tor_1^R(N, {}^{\vf^e}\!R)$ is torsion.
Since $F^e(\tf M)$ is torsion-free, we have  $\Tor_1^R(N,{}^{\vf^e}\!R) = 0$\,.   Now we invoke
Theorems~\ref{rmk:PS} and \ref{rmk:AM} to see that
\begin{equation*}
\Tor_i^R(N,{}^{\vf^e}\!R) = 0 \text{ for all } i\ge1\,,
\end{equation*}
and, moreover, that $N$ has finite projective dimension.   From \eqref{eq:push} it follows that $\Tor_i^R(\tf RM, {}^{\vf^e}\!R) = 0$ for all $i\ge 1$ and that $\tf RM$ has finite projective dimension.  Therefore we will have (2) once we show that $\tp RM = 0$.  For this, we apply $-\otimes_R{}^{\vf^e}\!R$ once again to \eqref{eq:tp}, to get an injection
\[
F^e(\tp RM)  \hookrightarrow F^e(M)\,.
\]
Since $F^e(\tp RM)$ is torsion and $F^e(M)$ is torsion-free, we have
$F^e(\tp RM)  = 0$.  If $\tp RM$ were non-zero, there would be a surjection $\tp RM \twoheadrightarrow R/\fm$.  But then  $F^e(R/\fm) =0$, that is, $\fm {{}^{\vf^e}\!R} = {}^{\vf^e}\!R$, an obvious contradiction, since $\fm {{}^{\vf^e}\!R} \subseteq \fm$.  Thus $\tp RM = 0$, and the proof that (1) $\implies$ (2) is complete.

Now assume (2) holds.
 Since $M$ is torsion-free, we can build the
universal pushforward \cite[\S1]{HJW}:
\begin{equation*}
0\lra M \lra R^{(\nu)} \lra N \lra 0\,,
\end{equation*}
where $\nu = \nu_R M^*$.
Then $N$ has finite projective dimension.
Now Theorem~\ref{rmk:PS} implies that
$\Tor_i^R(N, {}^{\vf^e}\!R) = 0$ for  all $i\ge1$.  Therefore
$\Tor_1^R(M, {}^{\vf^e}\!R) = 0$, and we get an injection
$F^e(M) \hookrightarrow ({}^{\vf^e}\!R)^{(\nu)}$, whence
$F^e(M)$ is torsion-free.  \end{proof}

From Theorem~\ref{rmk:PS} (alternatively, from the proof of Theorem~\ref{thm:carrier}),
we get  Tor-independence (item \eqref{pg:tor independent} in the introduction):

\begin{corollary}\label{cor:tor-van} If $R$ and $M$ satisfy the equivalent conditions of Theorem~\ref{thm:carrier}, then
$\Tor^R_i(M,{}^{\vf^{e'}}\!R) = 0$ for every $i\ge 1$ and every $e'\ge1$.
\end{corollary}

Of course, if $M$ is torsion-free, the converse of Corollary~\ref{cor:tor-van} holds, by Theorem~\ref{rmk:AM}.
In fact, it suffices to check that
$\Tor^R_i(M,{}^{\vf^{e'}}\!R) = 0$ for a single $e'$ and a single $i$.



\medskip

Recall that $R$ is
\emph{F-finite} provided ${\vf}$ is a finite map, that is, $R$ is
module-finite over ${\vf}(R)$.  In this case, $\vf^e$ is a finite map for each $e\ge1$.  Note
that the action of $R$ on the module $({}^{\vf^e}\!M)$ in items (1) and (2) below is
the {\em Frobenius} action $m \cdot r = m r^{p^e}$.

\begin{corollary}\label{cor:F-finite}
Assume that $(R,\fm)$ is a reduced local ring, is $F$-finite, and is a complete intersection. The
following conditions are equivalent:
\begin{enumerate}[\quad\rm(1)]
\item
$F^e({}^{\vf^{e'}}\!M)$ is torsion-free for every
  torsion-free $R$-module $M$ and every pair $e, e'$ of positive integers.
\item
$F^e({}^{\vf^{e'}}\!M)$ is torsion-free for some
  nonzero finitely generated $R$-module $M$ and some pair $e,e'$ of positive integers.
\item
$R$ is regular.
\end{enumerate}
\end{corollary}
\begin{proof}
Obviously (1)$\implies$(2), and the implication
(3)$\implies$(1) holds by Kunz's theorem~\cite[Theorem 2.1]{Ku} that
the $R$-module ${}^{\vf^{e}}{R}$ is flat when $R$ is regular.

To prove that (2)$\implies$(3), we note that $^{\vf^{e'}}\!M$  is a finitely generated $R$-module,
by F-finiteness.  Also, $^{\vf^{e'}}\!M$ is generically free because $R$ is reduced.  By Theorem~\ref{thm:carrier},
$^{\vf^{e'}}\!M$ has finite projective dimension, and now \cite[Theorem~1.1]{AHIY} implies  that $R$ is regular.
\end{proof}

\section*{Acknowledgments}
We would like to thank the referee for a critical reading of the paper.

\bibliographystyle{plain}

\begin{thebibliography}{10}


\bibitem{Au}
M.~Auslander,
\newblock Modules over unramified regular local rings,
\newblock \emph{Illinois J. Math.} \textbf{5} (1961), 631--647.


\bibitem{AHIY}
L.~L.~Avramov, M.~Hochster, S.~B.~Iyengar, and Y.~Yao,
\newblock Homological invariants of modules over contracting endomorphisms,
\newblock \emph{Math. Ann.} \textbf{353} (2012), 275--291.

\bibitem{AM}
L.~L.~Avramov and C.~Miller,
\newblock Frobenius powers of complete intersections,
\newblock \emph{Math. Res. Lett.} \textbf{8} (2001), 225--232.

\bibitem{BH}
W.~Bruns and J.~Herzog,
\newblock \emph{Cohen-{M}acaulay rings}, Cambridge Studies in Advanced Mathematics \textbf{39},
\newblock Cambridge University Press, Cambridge, 1993.


\bibitem{Her}    J.~Herzog,
\newblock Ringe der Charakteristik $p$ und Frobeniusfunktoren
\newblock \emph{Math. Z.} \textbf{140} (1974), 67--78.


\bibitem{HJW} C.~Huneke, D.~A. Jorgensen, and R.~Wiegand, \newblock
  Vanishing theorems for complete intersections, \newblock
  \emph{J. Algebra}, \textbf{238} (2001), 684--702.

\bibitem{HW1} C.~Huneke and R.~Wiegand, \newblock Tensor products of
  modules and the rigidity of Tor, \newblock \emph{Math. Ann.}
  \textbf{299} (1994), 449--476; \texttt{Correction:}
  \emph{Math. Ann.} \textbf{338} (2007), 291--293.

\bibitem{HW1e} C.~Huneke and R.~Wiegand, \newblock Correction to:
  ``{T}ensor products of modules and the rigidity of {T}or''
  [{M}ath. {A}nn. {\bf 299} (1994), 449--476], \newblock
  \emph{Math. Ann.}, (2007), 291--293.


\bibitem{HW2} C.~Huneke and R.~Wiegand, \newblock Tensor products of
  modules, rigidity and local cohomology, \newblock
  \emph{Math. Scand.} \textbf{81} (1997), 161--183.



\bibitem{Ku} E.~Kunz, \newblock{Characterization of regular local
  rings of characteristic $p$}, \newblock \emph{Amer. J. Math.}
  \textbf{91} (1969), 772--784.



\bibitem{Li}
S.~Lichtenbaum,
\newblock On the vanishing of {T}or in regular local rings,
\newblock \emph{Illinois J. Math.} \textbf{10} (1966), 220--226.

\bibitem{LW}
G.~J.~Leuschke and R.~Wiegand
\newblock\emph{{Cohen-Macaulay} {Representations}},
{Mathematical Surveys and Monographs},
{American Mathematical Society},
{Providence, RI},
{2012}.


\bibitem{MacLane} S.~Mac Lane \newblock \emph{Homology}, \newblock
  Reprint of the 1975 edition. Classics in
  Mathematics. Springer-Verlag, Berlin, 1995.


\bibitem{PS} C.~Peskine  and L.~Szpiro,
\newblock Dimension projective finie et cohomologie locale,
\newblock \emph{Inst. Hautes \'Etudes Sci. Publ. Math.}
\textbf{42} (1973), 47--119.


\end{thebibliography}

\end{document}